\documentclass[11pt,reqno]{amsart}
\usepackage{amsmath,enumerate,calc,float}
\newtheorem{theorem}{Theorem}
\newtheorem{conjecture}{Conjecture}
\newtheorem{question}{Question}
\newtheorem*{question*}{Question}
\newtheorem{lemma}[theorem]{Lemma}
\newtheorem{corollary}[theorem]{Corollary}
\newtheorem{proposition}[theorem]{Proposition}
\newcommand{\Alt}{A}
\newcommand{\Irr}{{\rm Irr}}
\newcommand{\mrow}{\theta}
\newcommand{\mcol}{\theta'}
\newcommand{\M}{\mathfrak m}
\newcommand{\len}{\mathfrak l}
\newcommand{\co}{\mathfrak f}
\newcommand{\Q}{\mathbb Q}
\begin{document}
\title{Zeros and roots of unity in character tables}
\dedicatory{Dedicated to the memory of Patrick X.\ Gallagher}
\author[A.~R.~Miller]{Alexander~Rossi~Miller}
\email{alexander.r.miller@univie.ac.at}
\begin{abstract}
  For any finite group $G$, Thompson proved that, for each  
  $\chi\in\Irr(G)$, $\chi(g)$ is a root of
  unity or zero for more than a third of the elements $g\in G$,
  and Gallagher proved that, for each larger than average
  class $g^G$, $\chi(g)$ is a root of unity or zero for
  more than a third of the irreducible characters $\chi\in\Irr(G)$.
  We show that in many cases ``more than a third'' can be replaced
  by ``more than half''.
\end{abstract}
\subjclass{20C15}\keywords{Zeros, roots of unity, characters}
\maketitle
\section{Introduction}
For any finite group $G$, let 
\[
  \mrow(G)=\min_{\chi\in \Irr(G)}\frac{|\{g\in G : \chi(g)\ \text{is a root of unity or zero}\}|}{|G|}
\]
and let
\[
  \mcol(G)=\min_{\text{$|g^G|\geq \frac{|G|}{|{\rm Cl}(G)|}$}}\frac{|\{\chi\in\Irr(G) : \chi(g)\ \text{is a root of unity or zero}\}|}{|\Irr(G)|}.
\]
Burnside proved that each $\chi\in\Irr(G)$ with $\chi(1)>1$ has at least one zero,
P.~X.~Gallagher proved that each $g\in G$ with $|g^G|>|G|/|{\rm Cl}(G)|$
is a zero of at least one $\chi\in\Irr(G)$, 
J.~G.~Thompson proved that 
\[
  \mrow(G)>1/3,
\]
and Gallagher proved that 
\[
  \mcol(G)>1/3.
\]
The proofs run by taking
the relations $\sum_{g\in G}|\chi(g)|^2=|G|$ ($\chi\in\Irr(G)$) and
$\sum_{\chi\in\Irr(G)}|\chi(g)|^2=|G|/|g^G|$, applying
the elements $\sigma$ of the Galois group
$\mathcal G={\rm Gal}(\Q(e^{2\pi i/|G|})/\Q)$, averaging 
over $\mathcal G$, and using
that the average over $\mathcal G$ of $|\sigma(\alpha)|^2$
is $\geq 1$ for any nonzero algebraic integer
$\alpha \in\mathbb Q(e^{2\pi i/|G|})$, or using the fact, due to C.~L.~Siegel,
that the average over $\mathcal G$ of $|\sigma(\alpha)|^2$
is $\geq 3/2$ for any algebraic integer
$\alpha \in\mathbb Q(e^{2\pi i/|G|})$ which is neither a root of unity nor zero,
cf.\ \cite{Burnside,Gallagher62,Gallagher,Siegel}.
For certain groups, there are also 
strong asymptotic results about zeros due to Gallagher, M.\ Larsen, and the author \cite{GallagherLarsenMiller,LarsenMiller,Miller}.

Are the lower bounds of $1/3$ for $\{\mrow(G): |G|<\infty\}$ and $1/3$ for $\{\mcol(G): |G|<\infty\}$
the best possible?
\begin{question}\label{Row:Question}
  What is the greatest lower bound of $\{\mrow(G) : |G|<\infty\}$?
\end{question}
\begin{question}\label{Column:Question}
  What is the greatest lower bound of $\{\mcol(G) : |G|<\infty\}$?
\end{question}

The author suspects that the answers to these questions are both $1/2$. In~particular, we propose the following:
\begin{conjecture}\label{General:Conjecture}
$\mrow(G)$ and $\mcol(G)$
are $\geq 1/2$ for every finite~group~$G$.
\end{conjecture}
We establish the conjecture for all finite nilpotent groups by establishing a much stronger result about zeros for this family of groups, which includes all $p$-groups.
The number of $p$-groups of order $p^n$ was
shown by G.~Higman~\cite{Higman} and C.~C.~Sims \cite{Sims}
to equal
$p^{\frac{2}{27}n^3+O(n^{8/3})}$ with $n\to\infty$, 
and it is a folklore conjecture that almost all finite groups are nilpotent in the sense that
\[
  \frac{\text{the number of nilpotent groups of order at most $n$}}{\text{the number of groups of order at most $n$}}=1+o(1),
\]
which, in view of our result, would mean that Conjecture~\ref{General:Conjecture} holds for almost all finite groups.

Conjecture~\ref{General:Conjecture} is readily verified 
for rational groups, such as Weyl groups, and
all groups of order $<2^9$, and although $\mrow(G)=1/2$ for
certain dihedral groups, the second inequality is strict in all known cases.
The author suspects that both inequalities are strict for all finite simple groups:
\begin{conjecture}\label{Simple:Conjecture}
 $\mrow(G)$ and $\mcol(G)$ are $>1/2$ for every finite~simple~group~$G$.
\end{conjecture}
We verify Conjecture~\ref{Simple:Conjecture} for 
$\Alt_n$, $L_2(q)$, $Suz(2^{2n+1})$, $Ree(3^{2n+1})$,
all sporadic groups, and all simple groups of order $\leq 10^9$.
We also show that both 
$\mrow(Suz(2^{2n+1}))$
and $\mcol(Suz(2^{2n+1}))$ tend to $1/2$ as $n\to\infty$.
In particular, the answers to Questions~\ref{Row:Question} and~\ref{Column:Question} must lie
between $1/3$ and~${1/2}$.

\section{Nilpotent groups}
We begin with our results on finite nilpotent groups.

\begin{theorem}\label{Zeros:Theorem}
  For each finite nilpotent group $G$,
  and each $\chi\in\Irr(G)$ with $\chi(1)>1$, 
  $\chi(g)=0$ for more than half of the elements $g\in G$. 
\end{theorem}

\begin{theorem}\label{Dual:Zeros:Theorem}
  Let $G$ be a finite nilpotent group, and let ${g\in G}$.\newline
  If ${|g^G|>\frac{|G|}{|{\rm Cl}(G)|}}$, then
  $\chi(g)=0$ for more than half of the nonlinear ${\chi\in\Irr(G)}$.
  If ${|g^G|=\frac{|G|}{|{\rm Cl}(G)|}}$,
  then $\chi(g)=0$ for at least half of the nonlinear $\chi\in\Irr(G)$.
\end{theorem}

\begin{corollary}\label{Main:Theorem}
$\mrow(G)$ and $\mcol(G)$ are $> 1/2$ whenever $G$ is nilpotent.
\end{corollary}

The key ingredient in the proof of Theorems~\ref{Zeros:Theorem} and \ref{Dual:Zeros:Theorem} is Theorem~\ref{nilpotent:Lemma}, which will replace
the result of Siegel used by Thompson and Gallagher.
Its proof relies on some auxiliary results of independent interest
and is based on arithmetic in cyclotomic fields.

For each positive integer $k$, we denote by $\zeta_k$
a primitive $k$-th root of unity.
For any algebraic integer $\alpha$ contained in some cyclotomic field,
we denote by $\len(\alpha)$ the least integer $l$ such that $\alpha$
is a sum of $l$ roots of unity, by $\co(\alpha)$ the
least positive integer $k$ such that $\alpha\in\Q(\zeta_k)$, and by
$\M(\alpha)$ the normalized trace
\[\frac{1}{[\Q(|\alpha|^2):\Q]}{\rm Tr}_{\Q(|\alpha|^2)/\Q}(|\alpha|^2),\]
so for any cyclotomic field $\Q(\zeta)$ containing $\alpha$,
\[\M(\alpha)=\frac{1}{|{\rm Gal}(\Q(\zeta)/\Q)|}\sum_{\sigma\in{\rm Gal}(\Q(\zeta)/\Q)} |\sigma(\alpha)|^2.\]

\begin{lemma}\label{coeff:cong}
  Let $a_1,a_2,\ldots,a_l$ and $b_1,b_2,\ldots,b_m$ be rational integers,
  and let $\alpha_1,\alpha_2,\ldots,\alpha_l$ and $\beta_1,\beta_2,\ldots,\beta_m$ be $p^n$-th roots of unity with
   $p$ prime and $n$ nonnegative.
  If 
  \[\sum_{j=1}^l a_j \alpha_j =\sum_{k=1}^m b_k \beta_k,\]
  then
  \[\sum_{j=1}^l a_j \equiv \sum_{k=1}^m b_k \pmod p.\]
\end{lemma}

\begin{proof}[Proof of Lemma~\ref{coeff:cong}]
  If $n=0$, then there is nothing to prove, so assume $n\geq 1$.
  Let $\zeta$ be a primitive $p^n$-th root of unity. For each $\alpha_j$ and $\beta_k$, let
  $r_j$ and $s_k$ be nonnegative integers such that $\alpha_j=\zeta^{r_j}$ and $\beta_k=\zeta^{s_k}$.
  Put
  \[
    P(x)=\sum_{j=1}^l a_jx^{r_j} - \sum_{k=1}^m b_kx^{s_k}.
  \]
  Then $P(\zeta)=0$, so $P(x)$ is divisible in $\mathbb Z[x]$ by
  the cyclotomic polynomial 
  \[\Phi_{p^n}(x)=\Phi_p(x^{p^{n-1}}).\]
  Hence $P(1)\equiv 0\pmod p$. 
\end{proof}

\begin{proposition}\label{Thm2}
  Let $G$ be a finite group, let $\chi\in\Irr(G)$, 
  and let $g$ be an element of $G$ with order a power of a prime $p$.
  If $p=2$ or $\chi(1)\not\equiv \pm 2\pmod p$, 
  then either $\chi(g)=0$, $\chi(g)$ is a root of unity, or $\M(\chi(g))\geq 2$.
\end{proposition}

\begin{proof}[Proof of Proposition~\ref{Thm2}]
  Suppose that $p=2$ or $\chi(1)\not\equiv \pm 2\pmod p$. 
  Let $p^n$ be the order of $g$, and let $\zeta$ be a primitive $p^n$-th root of unity, so 
  $\chi(g)\in \Q(\zeta)$.
  Let $\alpha=\zeta^m\chi(g)$ with $m$ such that 
  \begin{equation}\label{conduct:beta}
    \co(\alpha)=\min_k \co(\zeta^k\chi(g)).
  \end{equation}
  We will show that either $\alpha=0$, $\alpha$ is a root of unity, or
  $\M(\alpha)\geq 2$.

  Let $P=\co(\alpha)$. Using  
  $\Q(\zeta_k)\cap \Q(\zeta_l)=\Q(\zeta_{(k,l)})$, 
  then $P$ divides $p^n$. 
  If $P=1$, then $\alpha$ is rational and the conclusion follows.
  If $P$ is divisible by $p^2$, then for $\gamma$ a primitive $P$-th root of unity,
  $\alpha$ is uniquely of the shape
  \begin{equation}\label{beta:shape}
    \alpha=\sum_{k=0}^{p-1}\alpha_k\gamma^k,\quad\alpha_k\in\Q(\zeta_{P/p}),
  \end{equation}
  the $\alpha_k$ are algebraic integers, and 
  a straightforward calculation \cite[p.\ 115]{Cassels}
  shows that $\M(\alpha)$ is at least the number of nonzero
  $\alpha_k$. By \eqref{conduct:beta}, at least two of the
  $\alpha_k$ are nonzero. Hence $\M(\alpha)\geq 2$ if $p^2\mid P$. 

  It remains to consider the case $P=p$.
  Since $\Q(\zeta_2)=\Q(\zeta_1)$, we must have $p>2$.
  If $\len(\alpha)=0$, then $\alpha=0$; 
  if $\len(\alpha)=1$, then $\alpha$ is a root of unity;
  and if $\len(\alpha)>2$, then $\M(\alpha)\geq 2$ by a result of
  Cassels \cite[Lemma 3]{Cassels}.
  So assume $\len(\alpha)=2$.  Then by \cite[Thm.\ 1(i)]{Loxton},
  $\alpha$ can be written in the shape
  \[\alpha=\epsilon_1 \xi_1 +\epsilon_2 \xi_2,\quad \epsilon_k^2= 1,\]
  where $\xi_1$ and $\xi_2$ are $p$-th roots of unity. 
  If $\xi_1=\xi_2$, then either $\alpha=0$ or $\M(\alpha)= 4$.
  So assume
  \[
    \xi_1\neq \xi_2.
  \]
  By Lemma~\ref{coeff:cong},
  \begin{equation}\label{epsilons}
    \epsilon_1+\epsilon_2\equiv \chi(1)\pmod p.
  \end{equation}
  By \eqref{epsilons} and the fact that $\chi(1)\not\equiv \pm 2\pmod p$,
  \[\epsilon_1+\epsilon_2=0.\]
  Hence, for some root of unity $\rho$ and primitive $p$-th
  root of unity $\xi$, 
  \[\alpha=(\xi-1)\rho.\]
  Hence
  \[
    \M(\alpha)=\M(\xi-1)
    =2-\frac{1}{p-1}\sum_{k=1}^{p-1}( \xi^{k}+\xi^{-k})
    =2+\frac{2}{p-1}>2.
    \qedhere
  \]
\end{proof}

\begin{lemma}\label{Lemma:RootOfUnity}
  Let $G$ be a finite group, let $\chi\in\Irr(G)$, 
  and let $g$ be an element of $G$ with order a power of a prime $p$.
  If $\chi(1)\not\equiv\pm 1\pmod p$, then $\chi(g)$ is not a root of unity.
\end{lemma}

\begin{proof}[Proof of Lemma~\ref{Lemma:RootOfUnity}]
  Let $p^n$ be the order of $g$, so $\chi(g)\in \Q(\zeta_{p^n})$, and suppose 
  that $\chi(g)$ is a root of unity. Since the roots of unity in a given
  cyclotomic field $\mathbb Q(\zeta_k)$ are the $l$-th roots of unity for
  $l$ the least common multiple of $2$ and $k$, we then have
  \[\chi(g)=\epsilon\xi\]
  for some $\epsilon\in\{1,-1\}$
  and $p^n$-th root of unity $\xi$. 
  So by Lemma~\ref{coeff:cong}, either $\chi(1)\equiv 1\pmod p$ or
  $\chi(1)\equiv -1\pmod p$.
\end{proof}

\begin{lemma}\label{p:Lemma}
  Let $G$ be a finite group of prime-power order, let $g\in G$,
  and let $\chi\in\Irr(G)$.
  If $\chi(1)>1$, then either $\chi(g)=0$ or $\M(\chi(g))\geq 2$.
\end{lemma}

\begin{proof}[Proof of Lemma~\ref{p:Lemma}]
  If $|G|=p^n$ with $p$ prime, then each $g\in G$ has order a power of $p$, and
  each $\chi\in\Irr(G)$ has degree a power of $p$.
  So if $\chi(1)>1$, then by Proposition~\ref{Thm2} and
  Lemma~\ref{Lemma:RootOfUnity}, for each $g\in G$, 
  $\chi(g)=0$ or $\M(\chi(g))\geq 2$. 
\end{proof}

For any character $\chi$ of a finite group, let
\[\omega(\chi)=|\{\text{primes dividing $\chi(1)$}\}|.\]

\begin{theorem}\label{nilpotent:Lemma}
  Let $G$ be a finite nilpotent group, let $\chi\in\Irr(G)$, and let $g\in G$. Then
  \begin{equation}
    \chi(g)=0\quad\text{or}\quad\M(\chi(g))\geq 2^{\omega(\chi)}.
  \end{equation}
\end{theorem}

\begin{proof}[Proof of Theorem~\ref{nilpotent:Lemma}]
  If $|G|=1$, then $\chi(g)=\chi(1)=1$, so assume $|G|>1$. 
  Since $G$ is nilpotent, it is the direct product of its
  nontrivial Sylow subgroups $P_1,P_2,\ldots,P_n$.
  Let $g_1,g_2,\ldots,g_n$ be the unique sequence with $g_k\in P_k$ and
  \[g=g_1g_2\ldots g_n.\]
  For each $P_k$, let $\chi_k\in\Irr(P_k)$ be the unique irreducible
  constituent of the restriction of $\chi$ to $P_k$.
  Then
  \begin{equation}\label{Sylow:char:product}
    \chi(g)=\chi_1(g_1)\chi_2(g_2)\ldots \chi_n(g_n),\quad \chi(1)=\chi_1(1)\chi_2(1)\ldots \chi_n(1),
  \end{equation}
  \begin{equation}\label{dk:divides:Pk}
    \chi_k(1)\text{ divides }|P_k|,
  \end{equation}
  \begin{equation}\label{prime:Sylows}
    (|P_j|,|P_k|)=1\quad\text{for}\quad j\neq k,
  \end{equation}
  and 
  \begin{equation}\label{Sylow:Field}
    \chi_k(g_k)\in\mathbb Q(\zeta_{|P_k|}).
  \end{equation}
  For any algebraic integers $\alpha\in\mathbb Q(\zeta_l)$ and
  $\beta\in\mathbb Q(\zeta_m)$ with $(l,m)=1$,
  we have $\mathbb Q(\zeta_{lm})=\mathbb Q(\zeta_l)\mathbb Q(\zeta_m)$ and
  $\mathbb Q(\zeta_l)\cap\mathbb Q(\zeta_m)=\mathbb Q$, and hence
  \begin{equation}\label{Mult:M}
    \M(\alpha\beta)=\M(\alpha)\M(\beta).
  \end{equation}
  By \eqref{Sylow:char:product}, \eqref{prime:Sylows}, \eqref{Sylow:Field}, and \eqref{Mult:M},
  \begin{equation}\label{M:Product}
    \M(\chi(g))=\M(\chi_1(g_1))\M(\chi_2(g_2))\ldots\M(\chi_n(g_n)).
  \end{equation}
  By \eqref{M:Product} and Lemma~\ref{p:Lemma}, 
  \[\chi(g)=0\quad\text{or}\quad\M(\chi(g))\geq 2^w,\]
  where $w$ denotes the number of characters $\chi_k$ with $\chi_k(1)>1$.
  From \eqref{Sylow:char:product}, \eqref{dk:divides:Pk}, and \eqref{prime:Sylows},
  $w$ is equal to the number of prime divisors of $\chi(1)$. 
\end{proof}

\begin{proposition}\label{Nilpotent:Omega:Theorem}
  For each finite nilpotent group $G$, and each $\chi\in\Irr(G)$,
  \begin{equation}\label{Vanishing:Omega:Bound}
    \frac{|\{g\in G : \chi(g)=0\}|}{|G|}
    \geq
    1-\frac{1}{2^{\omega(\chi)}}\left(\frac{|G|-\chi(1)^2+2^{\omega(\chi)}}{|G|}\right).
  \end{equation}
\end{proposition}

\begin{proof}[Proof of Proposition~\ref{Nilpotent:Omega:Theorem}]
  Let $G$ be a finite nilpotent group,
  and let $\chi\in\Irr(G)$. 
  By Theorem~\ref{nilpotent:Lemma}, for each $g\in G$, 
  \begin{equation}\label{alt}
    \chi(g)=0\quad\text{or}\quad\M(\chi(g))\geq 2^{\omega(\chi)}.
  \end{equation}
  Now take the relation 
  \[
    |G|=\sum_{g\in G} |\chi(g)|^2,
  \]
  apply the elements $\sigma$ of the Galois group
  $\mathcal G={\rm Gal}(\Q(\zeta_{|G|})/\Q)$, and average over~$\mathcal G$.
  This gives
  \begin{equation}\label{row:M}
    |G|=\sum_{g\in G}\M(\chi(g)).
  \end{equation}
  From \eqref{alt} and \eqref{row:M}, 
  \begin{equation}\label{row:M:ineq}
    |G|\geq \chi(1)^2+2^{\omega(\chi)}|\{g\in G: \chi(g)\neq 0\}|-2^{\omega(\chi)}.
  \end{equation}
  By \eqref{row:M:ineq}, we have \eqref{Vanishing:Omega:Bound}.
\end{proof}

\begin{proof}[Proof of Theorem~\ref{Zeros:Theorem}]
  By Proposition~\ref{Nilpotent:Omega:Theorem}.
\end{proof}

\begin{proof}[Proof of Theorem~\ref{Dual:Zeros:Theorem}]
  Taking the relation 
  \[\frac{|G|}{|g^G|}=\sum_{\chi\in\Irr(G)} |\chi(g)|^2,\]
  applying the elements $\sigma$ of the Galois group
  $\mathcal G={\rm Gal}(\Q(\zeta_{|G|})/\Q)$, and averaging over~$\mathcal G$, 
  we have 
  \begin{equation}\label{col:M}
    \frac{|G|}{|g^G|}=\sum_{\chi\in \Irr(G)}\M(\chi(g)). 
  \end{equation}
  So for $\mathcal L=\{\chi\in\Irr(G): \chi(1)=1\}$ and
  $\mathcal N=\Irr(G)- \mathcal L$, 
  \begin{equation}\label{col:L:M}
    \frac{|G|}{|g^G|}= |\mathcal L|+\sum_{\chi\in \mathcal N}\M(\chi(g)). 
  \end{equation}
  By Theorem~\ref{nilpotent:Lemma}, for each $\chi\in \mathcal N$,
  \begin{equation}\label{column:pf:zero:two}
    \chi(g)=0\quad\text{or}\quad\M(\chi(g))\geq 2.
  \end{equation}
  From \eqref{col:L:M} and \eqref{column:pf:zero:two}, 
  \begin{equation}\label{L:N:Ineq}
    \frac{|G|}{|g^G|}\geq |\mathcal L|+2|\{\chi\in \mathcal N : \chi(g)\neq 0\}|.
  \end{equation}
  By \eqref{L:N:Ineq}, if $|{\rm Cl}(G)|=|G|/|g^G|$, then
  $|\{\chi\in \mathcal N : \chi(g)=0\}|\geq |\mathcal N|/2$,
  and if $|{\rm Cl}(G)|>|G|/|g^G|$, then
  $|\{\chi\in \mathcal N : \chi(g)=0\}|>|\mathcal N|/2$.
\end{proof}

\begin{proof}[Proof of Corollary~\ref{Main:Theorem}]
 By Theorem~\ref{Zeros:Theorem} and Theorem~\ref{Dual:Zeros:Theorem}.
\end{proof}

\section{Simple groups}
We now establish Conjecture~\ref{Simple:Conjecture} for several families of simple groups. 
\begin{theorem}\label{theorem simple groups}\ Let $n>0$.
  \begin{enumerate}[\rm I.]
  \item\label{Alt}
    For $G=\Alt_n$, we have
    \begin{equation}\label{Alt:Ineq}
      \mrow(G),\mcol(G)>
      \begin{cases} 1/2 & \text{if $n<9$,}\\
        3/4 & \text{if $n\geq 9$.}
      \end{cases}
    \end{equation}
    
  \item\label{Suz}  For $G=Suz(q)$ with $q=2^{2n+1}$, we have 
    \begin{equation}\label{suz:row}
      \mrow(G) = \frac{1}{2}+\frac{(q+1)(q^2+2)}{2q^2(q^2+1)}
    \end{equation}
    and
    \begin{equation}\label{suz:col}
      \mcol(G) = \frac{1}{2}+\frac{5}{2(q+3)},
    \end{equation}
    so $\mrow(G),\mcol(G)>1/2$ and
    \begin{equation}\label{Suz:Limits}
      \mrow(G),\mcol(G)\to 1/2\text{ as }q\to\infty.
    \end{equation}
    
  \item\label{L}
    For $G=L_2(q)$ with $q=p^n$ a prime power, we have $\mrow(G),\mcol(G)>1/2$.
    
  \item\label{Ree}  For $G=Ree(3^{2n+1})$, we have $\mrow(G),\mcol(G)>1/2$.
    
  \item\label{Spor} For each sporadic group $G$, we have $\mrow(G),\mcol(G)>1/2$.
    
  \item\label{Small} For each finite simple group $G$ of order $\leq 10^9$, we have
    ${\mrow(G),\mcol(G)>1/2}$.
  \end{enumerate}
\end{theorem}

\begin{corollary}\label{corollary}
$\inf\{\mrow(G): |G|<\infty\},\inf\{\mcol(G):|G|<\infty\}\in [1/3,1/2]$.
\end{corollary}
\begin{proof}[Proof of Corollary \ref{corollary}]
  Thompson and Gallagher give the lower bound of $1/3$. 
  The upper bound of $1/2$ follows from part \ref{Suz} of Theorem~\ref{theorem simple groups}.
\end{proof}

\begin{proof}[Verification of \ref{Alt}]
  \eqref{Alt:Ineq} holds up
  to $n=14$, so assume $n\geq 15$.
  In the character table of $\Alt_n$, the values
  are rational integers, except some values
  $\chi(g)$ with  
  \begin{equation*}
    |\chi(g)|^2=\frac{1+\lambda_1\lambda_2\ldots}{4}
  \end{equation*}
  for some partition $\lambda$ of $n$ into distinct odd parts
  $\lambda_1>\lambda_2>\dots$.  Since $n\geq 15$, it follows that 
  each pair $(\chi,g)\in\Irr(G)\times G$ satisfies
  \begin{equation}\label{alt:options}
    \chi(g)=0,\ |\chi(g)|=1,\ \text{or}\ |\chi(g)|^2\geq 4.
  \end{equation}
  Using~\eqref{alt:options} 
  and the fact that simple groups do not have irreducible characters of degree 2, 
  we get that each nonprincipal $\chi\in\Irr(G)$ satisfies
  \begin{equation}\label{rat:bound}
    |G|>|\{g\in G : |\chi(g)|=1\}|+4|\{g\in G : |\chi(g)|\neq 0,1\}|,
    \end{equation}
  and from \eqref{rat:bound} it follows that $\mrow(G)>3/4$. 
  Similarly, for any class $g^G$ with $|g^G|\geq |G|/|{\rm Cl}(G)|$, we have
  \[
    |{\rm Cl}(G)|
    \geq
    |\{\chi\in \Irr(G) : |\chi(g)|=1\}|
    +
    4|\{\chi\in \Irr(G) : |\chi(g)|\neq 0,1\}|,
  \]
  and hence $\mcol(G)>3/4$.
\end{proof}

\begin{proof}[Verification of \ref{Suz}]
Let $n\geq 2$, $r=2^n$, $q=2^{2n-1}$, and $G=Suz(q)$, so
\begin{equation*}
  |G|=q^2(q-1)(q^2+1)=q^2(q-1)(q-r+1)(q+r+1).
\end{equation*}
Maintaining the notation of Suzuki \cite{Suzuki},
there are elements $\sigma,\rho,\xi_0,\xi_1,\xi_2$ such that 
each element of $G$ can be conjugated into exactly one of the sets 
\[1^G,\sigma^G,\rho^G,(\rho^{-1})^G, A_0-\{1\},A_1-\{1\},A_2-\{1\},\]
where $A_i=\langle \xi_i\rangle $ ($i=1,2,3$), and
the irreducible characters of $G$ are given by the following table
\cite[Theorem 13]{Suzuki}:
\[
  \begin{matrix}
    & 1 & \sigma & \rho,\rho^{-1} & \xi_0^t\neq1 & \xi_1^t\neq1 & \xi_2^t\neq1 \\
    1 & 1 & 1  & 1& 1 & 1 & 1 \\
    X & q^2 & 0 & 0 & 1 & -1 & -1\\
    X_i & q^2+1 & 1 & 1 & \epsilon_0^i(\xi_0^t) & 0 & 0 \\
    Y_j & (q-r+1)(q-1) & r-1 & -1 & 0 & -\epsilon_1^j(\xi_1^t) & 0\\
    Z_k & (q+r+1)(q-1) & -r-1 & -1 & 0 & 0 & -\epsilon_2^k(\xi_2^t) \\
    W_l &\frac{r(q-1)}{2} & -\frac{r}{2} & \pm \frac{r\sqrt{-1}}{2} & 0 & 1 & -1
  \end{matrix}
\]
\noindent Here, $1\leq i\leq \frac{q}{2}-1$, $1\leq j\leq \frac{q+r}{4}$,
$1\leq k\leq \frac{q-r}{4}$, $1\leq l\leq 2$, 
\begin{equation}\label{epsilon0}
  \epsilon_0^i(\xi_0^t)=\zeta^{it}+\zeta^{-it},\quad \zeta=e^{2\pi \sqrt{-1}/(q-1)},
\end{equation}
and $\epsilon_1^j$ and $\epsilon_2^k$ are certain characters on $A_1$ and
$A_2$. The $A_i$'s satisfy 
\begin{equation}\label{Suzuki:Ai's}
  |A_0|=q-1,\quad |A_1|=q+r+1,\quad |A_2|=q-r+1,
\end{equation}
and
denoting by $G_i$ ($i=0,1,2$) the set of elements $g\in G$
that can be conjugated into $A_i-\{1\}$, we have
\begin{equation}\label{Suzuki:Gi's}
  |G_i|=\frac{|A_i|-1}{l_i}\frac{|G|}{|A_i|},
\end{equation}
where $l_0=2$ and $l_1=l_2=4$.

Let $\gamma_s=\zeta^s+\zeta^{-s}$ with
$\zeta=e^{2\pi \sqrt{-1}/(q-1)}$ and $s\in\mathbb Z$.
Then
\[
  |\gamma_s|=1\Leftrightarrow
  6s\pm (q-1)\equiv 0 \pmod{3(q-1)}
\]
and
\[
  \gamma_s=0\Leftrightarrow 
  4s\pm (q-1) \equiv 0\pmod{2(q-1)}.
\]
Since $q-1\equiv 1 \pmod{3}$, and $q-1\equiv 1 \pmod{2}$, it follows that
\begin{equation}
  |\gamma_s|\not\in\{0,1\}\text{ for all }s\in \mathbb Z.
\end{equation}
So for any $X_i$, 
\begin{equation}\label{X1}
  |\{g\in G : X_i(g)\ \text{is a root of unity or zero}\}| = |G|-|G_0|-1, 
\end{equation}
and for any $g\in G_0$, 
\begin{equation}\label{Xi0}
  |\{\chi\in \Irr(G) : \chi(g)\ \text{is a root of unity or zero}\}|=\frac{q}{2}+4.
\end{equation}

By~\eqref{X1} and \eqref{Suzuki:Ai's}--\eqref{Suzuki:Gi's}, 
\begin{equation}\label{mr:bound}
  \mrow(G)\leq \frac{1}{2}+\frac{(q+1)(q^2+2)}{2q^2(q^2+1)}.
\end{equation}
Equality must hold in \eqref{mr:bound} because 
\[
  |\{g\in G : W_l(g)\in\{0,1,-1\}|=|G_0|+|G_1|+|G_2|>|G|-|G_0|-1
\]
and for any $\chi\in\Irr(G)-\{X_i\}-\{W_l\}$,
\[
  |\{g\in G : \chi(g)\in\{0,1,-1\}\}|\geq 2|\rho^G|+|G_0|+|G_2|>|G|-|G_0|-1.
\]

By~\eqref{Xi0} and the fact that, for any $g\in G_0$, 
$|C_G(g)|=q-1<q+3=|{\rm Cl}(G)|$, we have 
\begin{equation}\label{mc:bound}
  \mcol(G)\leq \frac{1}{2}+\frac{5}{2(q+3)}.
\end{equation}
Equality must hold in \eqref{mc:bound} because 
$1^G$, $\sigma^G$, and $\rho^G$ have size $<|G|/|{\rm Cl}(G)|$, 
and for any $g\in G_1\cup G_2$, 
\[
  |\{\chi\in \Irr(G) : \chi(g)\in\{0,1,-1\}\}|\geq \frac{3q-r+12}{4}\geq\frac{q}{2}+4.
  \qedhere
\]
\end{proof}

\begin{proof}[Verification of \ref{L}]
Let $q=p^n$ with $p$ prime, $G=L_2(q)$, let $R$ and $S$ be as in 
\cite[pp.~402--403]{Jordan}, and let $G_0$ (resp.\ $G_1$)
be the set of nonidentity elements $g\in G$ that
can be conjugated into $\langle R\rangle$ (resp.\ $\langle S\rangle$).

Assuming first $p\neq 2$, then 
\begin{alignat}{2}
  |G|&=\frac{q(q^2-1)}{2},&\quad |{\rm Cl}(G)|&=\frac{q+5}{2},\\
  |G_0|&=\frac{q(q+1)(q-3)}{4},&\quad |G_1|&=|G_0|+q=\frac{q(q-1)^2}{4},
\end{alignat}
and $G-G_0\cup G_1$ consists of 3 classes: $1^G$, $a^G$, $b^G$,
with  $|C_G(a)|=|C_G(b)|=q$. 
Inspecting Jordan's table \cite[p.\ 402]{Jordan}, each $\chi\in\Irr(G)$ satisfies either
\begin{enumerate}[(i)]
\item $\chi(g)\in \{0,1,-1\}$ on $G_0\cup G_1$, or
\item $\chi(g)\in\{1,-1\}$ on $a^G\cup b^G$ and $\chi(g)=0$ on   
  $G_0$ or $G_1$.
\end{enumerate}
If $q>3$, then
\begin{equation}\label{L2:bds}
  |G_0|+|G_1|>|G|/2\quad\text{and}\quad
  |a^G|+|b^G|+|G_0|>|G|/2,
\end{equation}
and if $q=3$, then $G\cong A_4$. So $\mrow(G)>1/2$.
Similarly, 
\[
  |\{\chi\in\Irr(G): \chi(a),\chi(b)\in\{0,1,-1\}\}|=\frac{q+1}{2},
\]
and for $g\in G_0$ (resp.\ $g\in G_1$) and $\chi\in\Irr(G)$, we have
$\chi(g)\in\{0,1,-1\}$ away from the 
$\leq (q-3)/4$ irreducible characters of degree $q+1$ (resp.\
the $\leq (q-1)/4$ characters of degree $q-1$), from which it follows that 
$\mcol(G)>1/2$.

For $p=2$, we have
$|G|=q(q^2-1)$, $|{\rm Cl}(G)|=q+1$, 
\begin{equation}
  |G_0|=\frac{q(q+1)(q-2)}{2},\quad 
  |G_1|=\frac{q^2(q-1)}{2},
\end{equation}
and $G-G_0\cup G_1$ consists of 2 classes: $1^G$ and $a^G$ with $|C_G(a)|=q$.
The irreducible characters of $G$ are given by Jordan \cite[p.\ 403]{Jordan}. 
There is 
the principal character,
$1$ character of degree $q$, $q/2$ 
characters of
degree $q-1$, and $q/2-1$ characters of degree $q+1$.
All the characters satisfy $\chi(g)\in\{0,1,-1\}$ on $a^G$,
the character of degree $q$ is $\pm 1$ on $G_0$ and $G_1$,
the characters of degree $q-1$ vanish on $G_0$, and
the characters of degree $q+1$ vanish on $G_1$. From this, it follows
that $\mrow(G)$ and $\mcol(G)$ are $>1/2$.
\end{proof}

\begin{proof}[Verification of \ref{Ree}] 
Let $n$ be a positive integer, $m=3^n$, $q=3^{2n+1}$, and ${G=Ree(q)}$,
so
\[
  |G|=q^3(q-1)(q+1)(q^2-q+1),\quad 
  |{\rm Cl}(G)|=q+8.
\]
The irreducible characters of $G$ are given by Ward~\cite{Ward}
in a 16-by-16 table, with the last 6 rows being occupied
by 6 families of exceptional characters, the sizes of which are,
from top to bottom, 
\[
  \frac{q-3}{4},\ \frac{q-3}{4},\ \frac{q-3}{24},\ \frac{q-3}{8},\
  \frac{q-3m}{6},\ \frac{q+3m}{6}.
\]
From Ward's table, we find that 
for any class $g^G\not\in\{1^G,X^G,J^G\}$, 
$\chi(g)\in\{ 0,1,-1\}$ for more than half of the
irreducible characters $\chi$ of $G$.
Since the classes $1^G,X^G,J^G$ all have size $<|G|/|{\rm Cl}(G)|$,
we conclude that $\mcol(G)>1/2$.
  
The first step in verifying $\mrow(G)>1/2$ 
is to compute the following table:

\begin{table}[H]
  \caption{}\label{dist}
  {\renewcommand{\arraystretch}{1.5}\begin{tabular}{@{\hspace{.03\textwidth}}l@{\hspace{.1\textwidth}}c@{\hspace{.03\textwidth}}}
$\Omega\subset G$ & $|\cup_{g\in G} g{\Omega}g^{-1}|$  \\
\hline $\{1\}$ & 1\\
$\langle R\rangle-\{ 1\}$  & $\frac{q^3(q-3)(q^3+1)}{4}$\\
$\langle S\rangle -\{1\}$ & $\frac{q^3(q-1)(q-3)(q^2-q+1)}{24}$\\
${M^-}-\{1\}$ & $\frac{q^3(q-1)(q+1)(q^2-2q-3m)}{6}$\\
$M^+-\{ 1\}$ & $\frac{q^3(q-1)(q+1)(q^2-2q+3m)}{6}$\\
$\{X\}$ & $|G|/q^3$\\ 
$\{Y\}$ & $|G|/3q$\\ 
$\{T\}$ & $|G|/2q^2$\\ 
$\{T^{-1}\}$ & $|G|/2q^2$\\ 
$\{YT\}$ & $|G|/3q$\\ 
$\{YT^{-1}\}$ & $|G|/3q$\\ 
$\{JT\}$ & $|G|/2q$\\ 
$\{JT^{-1}\}$ & $|G|/2q$\\ 
$J\langle R\rangle- \{J\}$ & $\frac{q^3(q-3)(q^3+1)}{4}$\\
$J\langle S\rangle -\{J\}$ & $\frac{q^3(q-1)(q-3)(q^2-q+1)}{8}$\\
$\{J\}$ & $|G|/q(q^2-1)$\\
\end{tabular}}
\end{table}

Then with Table~\ref{dist} and Ward's table in hand,
a straightforward check establishes that, for
each $\chi\in\Irr(G)$, 
\[
  |\{g\in G : \chi(g)\in\{0,1,-1\}\}|>|G|/2.
\]
Hence $\mrow(G)>1/2$. 
\end{proof}

\subsubsection*{Verification of \ref{Spor} and \ref{Small}}
Here, in Tables~\ref{sporadic} and \ref{small},
we report the values of $\mrow$ and $\mcol$ for
sporadic groups and simple groups of order $\leq 10^9$.
All values are rounded to the number of digits shown.

\begin{center}
\begin{table}[H]
\caption{The sporadic groups.\label{sporadic}}
\begin{tabular}{|@{\hspace{.03\textwidth}}p{{.12\textwidth}}c@{\hspace{.06\textwidth}}c@{\hspace{.03\textwidth}}|}
  \hline  $G$ & $\mrow(G)$ & $\mcol(G)$  \\
  \hline
  $M_{11}$ & 0.7290 & 0.8000 \\
  $M_{12}$ & 0.7955 & 0.8667 \\
  $M_{22}$ & 0.7117 & 0.8333 \\
  $M_{23}$ & 0.6827 & 0.7647 \\
  $M_{24}$ & 0.6913 & 0.7692 \\
  $J_1$   & 0.5583 & 0.6000 \\
  $J_2$   & 0.6373 & 0.6190 \\
  $J_3$   & 0.5840 & 0.7143 \\
  $J_4$   & 0.6925 & 0.7903 \\
  $Co_1$  & 0.8739 & 0.8515 \\
  $Co_2$  & 0.8347 & 0.8333 \\
  $Co_3$  & 0.7528 & 0.8333 \\
  $Fi_{22}$ & 0.8029 & 0.8769 \\
  \hline
\end{tabular}\quad
\begin{tabular}{|@{\hspace{.03\textwidth}}p{{.12\textwidth}}c@{\hspace{.06\textwidth}}c@{\hspace{.03\textwidth}}|}
\hline  $G$ & $\mrow(G)$ & $\mcol(G)$  \\
\hline
  $Fi_{23}$ & 0.8328 & 0.8469 \\
  $Fi_{24}'$ & 0.8808 & 0.8056\\
  $HS$    & 0.7853 & 0.8750 \\
  $McL$   & 0.6722 & 0.8333 \\
  $He$    & 0.7088 & 0.7576\\
  $Ru$    & 0.8517 & 0.8333 \\
  $Suz$   & 0.8141 & 0.8372 \\
  $O\text{'}N$    & 0.6830 & 0.8667\\
  $HN$    & 0.6362 & 0.7593 \\
  $Ly$    & 0.7879 & 0.8491 \\
  $Th$    & 0.7978 & 0.8750 \\
  $B$     & 0.8812 & 0.8587 \\
  $M$     & 0.8855 & 0.8711 \\
  \hline
\end{tabular}
\end{table}\end{center}

\begin{center}
\begin{table}[H]
  \caption{The simple groups of order $\leq 10^9$
    that are not cyclic, 
    $\Alt_n$, $L_2(q)$, $Suz(2^{2n+1})$, $Ree(3^{2n+1})$, or sporadic. 
    \label{small}}
\begin{tabular}{|@{\hspace{.03\textwidth}}p{.12\textwidth}c@{\hspace{.06\textwidth}}c@{\hspace{.03\textwidth}}|}
  \hline
  $G$ & $\mrow(G)$ & $\mcol(G)$ \\
  \hline
  $L_3(3)$ & 0.6736 & 0.8333 \\
  $U_3(3)$ & 0.7049 & 0.8571 \\ 
  $L_3(4)$ & 0.6000 & 0.8000 \\
  $S_4(3)$ &  0.8713 & 0.9000 \\
  $U_3(4)$ & 0.6892 & 0.7273 \\
  $U_3(5)$ & 0.7103 & 0.8571 \\
  $L_3(5)$ & 0.6754 & 0.8667 \\
  $S_4(4)$ & 0.6433 & 0.7037 \\
  $S_6(2)$ & 0.8867 & 0.8333\\
  $L_3(7)$ & 0.6235 & 0.7273 \\
  $U_4(3)$ & 0.7121 & 0.9000 \\
  $G_2(3)$  & 0.8321 & 0.9130 \\
  $S_4(5)$ & 0.6501 & 0.6471 \\
  $U_3(8)$ & 0.5701 & 0.6786 \\
  $U_3(7)$ & 0.6741 & 0.7586 \\
  $L_4(3)$ & 0.6911 & 0.8621 \\
  \hline
\end{tabular}\quad 
\begin{tabular}{|@{\hspace{.03\textwidth}}p{{.12\textwidth}}c@{\hspace{.06\textwidth}}c@{\hspace{.03\textwidth}}|}
  \hline
  $G$ & $\mrow(G)$ & $\mcol(G)$ \\
  \hline
  $L_5(2)$ & 0.7038 & 0.7778 \\
  $U_5(2)$ & 0.8041 & 0.9149 \\
  $L_3(8)$ & 0.6650 & 0.7083 \\
  $^2F_4(2)'$ & 0.7006 & 0.8182 \\
  $L_3(9)$ & 0.5488 & 0.6000 \\
  $U_3(9)$ & 0.6237  & 0.6739  \\
  $U_3(11)$ & 0.5494 & 0.6250 \\
  $S_4(7)$ & 0.7341 & 0.7308 \\
  $O_8^+(2)$ & 0.8555  & 0.9245  \\
  $O_8^-(2)$ & 0.7578  &  0.8462 \\
  $^3D_4(2)$ & 0.6920  & 0.6571 \\
  $L_3(11)$ & 0.6660  & 0.6970 \\
  $G_2(4)$ & 0.6449  &  0.7500 \\
  $L_3(13)$ & 0.5354 & 0.5938 \\
  $U_3(13)$ & 0.6662  &  0.6957 \\
  $L_4(4)$ & 0.6020  & 0.5714  \\
  \hline
\end{tabular}
\end{table}
\end{center}

\end{document}